\documentclass{aims}

\usepackage{amsmath}
  \usepackage{paralist}
  \usepackage{subfigure}
  \usepackage{graphics} 
  \usepackage{epsfig} 
\usepackage{graphicx}  \usepackage{epstopdf}
 \usepackage[colorlinks=true]{hyperref}
\hypersetup{urlcolor=blue, citecolor=red}

    \def\ppages{X--XX}
     \def\DOI{10.3934/xx.xx.xx.xx}

\newcommand{\vect}{{\operatorname {vec}}}

\newcommand{\vech}{{\operatorname{vech}}}

\newcommand{\vechinv}{{\operatorname{vech}^{-1}}}

\newcommand{\trace}{{\mathrm{Tr }}}

\newcommand{\RR}{\mathbb{R}}

\newcommand{\CC}{\mathbb{C}}

\newcommand{\LLL}{\mathcal{L}}

\newcommand{\diag}{\operatorname{diag}}

\newcommand{\norm}[1]{||#1||}

\newtheorem{theorem}{Theorem}[section]

\newtheorem{lemma}[theorem]{Lemma}

\theoremstyle{definition}
\newtheorem{definition}[theorem]{Definition}

\newtheorem{thm}{Theorem}

\title[Convergence of SCF using a density matrix approach] 
      { A density matrix approach to   the convergence  of the self-consistent field iteration}

\author[Parikshit Upadhyaya, Elias Jarlebring, Emanuel Rubensson]{}


 \email{pup@kth.se}
 \email{eliasj@kth.se}
 \email{emanuel.rubensson@it.uu.se}

\begin{document}
\maketitle

\centerline{\scshape Parikshit Upadhyaya}
\medskip
{\footnotesize
 \centerline{Lindstedtsv\"agen 25,}
   \centerline{Department of Mathematics,}
   \centerline{SeRC - Swedish e-Science research center,}
   \centerline{Royal Institute of Technology, SE-11428 Stockholm, Sweden}
} 

\medskip

\centerline{\scshape Elias Jarlebring}
\medskip
{\footnotesize
 \centerline{Lindstedtsv\"agen 25,}
 \centerline{Department of Mathematics,}
   \centerline{SeRC - Swedish e-Science research center,} 
   \centerline{Royal Institute of Technology, SE-11428 Stockholm, Sweden}
}

\medskip
\centerline{\scshape Emanuel H. Rubensson}
\medskip
{\footnotesize
   \centerline{Division of Scientific Computing, Department of Information Technology,}
   \centerline{Uppsala University, Box 337, SE-75105 Uppsala, Sweden}
}



\begin{abstract}
In this paper, we present a local convergence analysis of the self-consistent field (SCF) iteration using the density matrix as the state of a fixed-point iteration. Sufficient and almost necessary conditions for local convergence are formulated in terms of the spectral radius of the Jacobian of a fixed-point map. The relationship between convergence and certain properties of the problem is explored by deriving upper bounds expressed in terms of higher gaps. This gives more information regarding how the gaps between eigenvalues of the problem affect the convergence, and hence these bounds are more insightful on the convergence behaviour than standard convergence results. We also provide a detailed analysis to describe the difference between the bounds and the exact convergence factor for an illustrative example. Finally we present numerical examples and compare the exact value of the convergence factor with the observed behaviour of SCF, along with our new bounds and the characterization using the higher gaps. We provide heuristic convergence factor estimates in situations where the bounds fail to well capture the convergence.
\end{abstract}

\section{Introduction}\label{sect:intro}
Let $A:M\to M$, where  $M\subset\CC^{n\times n}$ denotes the set of Hermitian matrices. In this work we consider the associated nonlinear
eigenvalue problem consisting of determining $(X_1,\Lambda_1)\in\CC^{n\times p}\times \RR^{p\times p}$
such that $(X_1,\Lambda_1)$ is an invariant pair of $A(X_1X_1^H)$, i.e.,
  \begin{subequations}\label{eq:prob}
  \begin{eqnarray}
A(X_1X_1^H)X_1&=&X_1\Lambda_1 ,\\
    X_1^HX_1&=& I,
  \end{eqnarray}
  \end{subequations}
where 
$\Lambda_1=\diag(\lambda_1,\ldots,\lambda_p)$ and
$\lambda_1,\ldots,\lambda_n$ are
the eigenvalues of $A(X_1X_1^H)$, numbered in ascending order.
%
This is one of the fundamental computational challenges in quantum chemistry and related fields (Hartree-Fock and Kohn-Sham density functional theory).\footnote{In these settings, the columns of $X_1$ contain the basis-expansion coefficients of the molecular orbitals.} However, this problem also arises as a trace ratio maximization problem in linear discriminant analysis for dimension reduction. See \cite{ngosaadtrp},\cite{zhangliaotrp} and \cite{zhanganotetrp} for more on this application. The self-consistent field (SCF) iteration consists of computing iterates satisfying
the linear eigenvalue problem
\begin{equation} \label{eq:scf1}
   A(V_kV_k^H)V_{k+1}=V_{k+1}S_{k+1}, 
\end{equation}
where $V_{k+1}^HV_{k+1}=I$ and $S_{k+1}\in\RR^{p\times p}$ is diagonal. In the case of convergence, $V_{k} \to X_1$ and $S_k \to \Lambda_1$.  In this paper we provide a local convergence analysis of this algorithm. SCF is rarely used on its own as a solution method and most of the state-of-the-art procedures are based on its enhancements and improvements, for example, Pulay's DIIS (Direct Inversion in the Iterative Subspace) acceleration in \cite{Pulay:1979}. See also standard references \cite{Szabo:1996:QC},\cite{Helgaker:Jorgensen:Olsen} and further literature discussion below.

There is an extensive amount of literature on the convergence
of the SCF iteration and its variants.
We mention some main approaches to the convergence theory,
without an ambition of a complete description. A number of recent
works are based on the optimization viewpoint, e.g.,
\cite{Cances:2000:SCF,Levitt:2012:CONVERGENCE,liu2014convergence,liu2015analysis}. This is natural, since the problem in \eqref{eq:prob} often stems from the first order optimality condition of an energy minimization problem, as in \cite[Section~2.1]{Saad:2010:ELECSTRUCT}. In particular, the Roothaan algorithm with level-shifting and damping are studied in \cite{Cances:2000:SCF}. This analysis was used as a basis for the gradient analysis in \cite{Levitt:2012:CONVERGENCE}, which provided
explicit estimates of the convergence rate for the algorithms applied to the  Hartree-Fock equations. The convergence of the DIIS acceleration scheme has been studied separately in \cite{ROHWEDDER:2011}.

Various approaches are based on measuring the subspace angle and
other using chordal norms, e.g.,  \cite{liu2014convergence,Cai:2017:EIGVEC},
leading to local convergence as well as a global convergence analysis.
In contrast to these approaches, we use a density matrix
based analysis and derive bounds involving higher gaps (as we explain below).
The analysis in \cite{liu2015analysis} provides precise
conditions for local convergence (and some global convergence conditions),
under the assumption that $A$ only depends on the diagonal of the
density matrix $X_1X_1^H$, which in many discretization settings corresponds to the charge density. Not all
problems are nonlinear only in the diagonal of the density matrix, as e.g., 
the example in Section~\ref{subsec:watermolecule}.
The work in \cite{Yang:2009:SCF} also provides a convergence
analysis, mostly based on
a non-zero temperature filter function approach. 
We note also that
a precise local convergence criterion for the classical version
was presented in \cite{Stanton:1981:CLOSEDSHELL}, not
involving a density matrix analysis. 

A model of interacting bosons which has received considerable attention is the Gross-Pitaevskii equation. This corresponds to \eqref{eq:prob} with $p = 1$. Convergence results for the SCF iteration
for this case can be found for example in \cite{Bai_LV:2017}.

Our convergence analysis is focused on establishing a precise characterization of
the convergence factor as well as natural upper bounds. We provide an exact formula for the 
convergence factor, which turns out to be the spectral radius of a matrix (the Jacobian of the fixed-point map). Using this exact formula,
we derive upper bounds which can be phrased in terms of higher gaps (as we define later in Definition~\ref{def:hogap}) and the action
of a linear operator on the outer products of eigenvectors. We also provide an example where the convergence cannot be characterized based 
on the first gap alone, which illustrates the importance  of taking into account the higher gaps in the convergence analysis. This should be viewed in contrast to the analysis in \cite{liu2014convergence,Cai:2017:EIGVEC,Yang:2009:SCF}, which is primarily focused on the first gap.

We will use a formulation of the SCF iteration in terms of the density matrix. In our context, a density matrix
is defined by 
\[
  P_{k}:=V_{k}V_{k}^H\in M.
\]
Given $P_k$ we can compute $A(P_k)=A(V_kV_k^H)$ from
which we can compute $V_{k+1}$ and in 
principle construct $P_{k+1}=V_{k+1}V_{k+1}^H$. Hence,
the iteration \eqref{eq:scf1} is equivalent
to a fixed point iteration in $P_k$. We will
refer to this fixed point map as $\Psi$, i.e.,
\begin{equation}\label{eq:psifp}
  P_{k+1}=\Psi(P_k).
\end{equation}
Although our conclusions hold for general problems, we restrict our analysis to the case where the operator $A$ has the form
\begin{equation}\label{eq:ap}
A(P)=A_0+\LLL(P),
\end{equation}
where $A_0\in M$ and $\LLL:\CC^{n\times n}\to \CC^{n\times n}$ is a complex linear operator, i.e., $\LLL(zA)=z\LLL(A)$ for all $z\in\CC$. The density matrix formulation in \eqref{eq:psifp} is the starting point of several linear scaling variants of the SCF-algorithm that avoid explicit construction of $V_{k+1}$\cite{BOWLER:MIZAYAKI}.  
Before we proceed to the next section, we will introduce some necessary notation and definitions. 

Let $X,\Lambda$ correspond to a complete eigenvalue
decomposition of $A$ evaluated in a solution to \eqref{eq:prob}, i.e., 
\[
  A(X_1X_1^H)X=X\Lambda,
\]
where
\begin{equation}\label{eq:lambda}
 \Lambda=\begin{bmatrix}\Lambda_1 & \\ & \Lambda_2\end{bmatrix}
=\diag(\lambda_1,\ldots,\lambda_p,\lambda_{p+1},\ldots,\lambda_n)
\end{equation}
and
\[
 X=[X_1\;\;X_2]=[x_1\;\;\ldots\;\;x_n].
\]

\begin{definition}[Gap]
The smallest distance between the diagonal elements
of $\Lambda_1$ and the diagonal elements of $\Lambda_2$ is defined as 
as the \emph{gap} and denoted $\delta$.
\end{definition}
Due to the numbering of eigenvalues, the gap is given by
\[
  \delta=\min_{i\le p,\;j\ge p+1}|\lambda_i-\lambda_j|=\lambda_{p+1}-\lambda_p.
\]
Throughout this paper, we assume that $\delta \neq 0$. Otherwise the decomposition \eqref{eq:lambda} is not unique and violates the common uniform well-posedness hypothesis of \cite{Cances:2000:SCF}.
\begin{definition}[Higher gap]\label{def:hogap}
The $j$-th smallest distance between the diagonal elements of $\Lambda_1$ and $\Lambda_2$ is denoted by $\delta_j$. 
\end{definition}
Note that $\delta_1 = \delta$. As an example, the second gap is given by
\[
 \delta_2=\underset{(i,j)\neq (p,p+1)}{\min_{i\le p,\;j\ge p+1}}|\lambda_i-\lambda_j|
=\min(\lambda_{p+2}-\lambda_p,\lambda_{p+1}-\lambda_{p-1}).
\]
\begin{definition}[Lower triangular vectorization]
Let $m = n(n+1)/2$. The operator $\mathit{\vech} \colon M \to \CC^{m}$ is defined as
\begin{equation*}
  \vech(W) = \begin{bmatrix}
              w_{1,1}&
              \cdots&
              w_{n,1}&
              w_{2,2}&
              \cdots &
              w_{n,2}&
              w_{3,3}&
              \cdots &
              w_{n,n}
              \end{bmatrix}^T.
\end{equation*}
\end{definition}
This is the vectorization operator adapted for hermitian matrices, and returns the vectorizaton of the lower triangular part. 
Similarly, we define the inverse operator $\vech^{-1} \colon \CC^{m} \to M$ which maps any vector $v \in \CC^m$ to a corresponding  $W \in M$.  The relation between $\vect$ and $\vech$ is
given by 
\begin{equation}\label{eq:vvech}
\vech(W)=T\vect(W),
\end{equation}
where $T\in\RR^{m\times n^2}$. The matrix $T$ is in general non-unique (as discussed in \cite{McCull:SymMat} and \cite{Henderson:vech}). In this paper, we will specifically use \eqref{eq:vvech} with
\begin{equation*}
T = \diag(I_n,\begin{bmatrix}0& I_{n-1}\end{bmatrix},\begin{bmatrix}0& 0& I_{n-2}\end{bmatrix},\cdots,1).
\end{equation*}
\section{Convergence characterization}
\subsection{Main theory}
The following characterization involves the matrix
consisting of reciprocal gaps, which we denote  $R\in\RR^{n\times n}$,
and is given by
\begin{equation}\label{eq:Rdef}
R_{i,j}=
\begin{cases}
\frac{1}{\lambda_j(0)-\lambda_i(0)},\;\;& \textrm{ if }i\le p \textrm{ and } j>p\\
\frac{1}{\lambda_i(0)-\lambda_j(0)},\;\;& \textrm{ if }i> p \textrm{ and } j\le p\\
0&\textrm{ otherwise}.
\end{cases}
\end{equation}
where $\lambda_1(t),\ldots,\lambda_n(t)$ are eigenvalues
of a parameter dependent matrix $B(t)$.\footnote{Note that the matrix $R$ has appeared with different names in other papers before. For example, in \cite{Stanton:1981:CLOSEDSHELL}, it is referred to as the "density perturbation". In \cite{liu2015analysis}, it is called "first divided difference matrix".} The matrix $R$ is symmetric with the following structure:
\begin{equation*}
  R = \begin{bmatrix}0 &R_p^T\\R_p& 0\end{bmatrix},
\end{equation*}
where $R_p\in \mathbb{R}^{(n-p)\times p}$. We need the following perturbation result whose variants exist in quantum mechanical perturbation theory, for example in chapter 15.III of \cite{Messiah1999}.

\begin{lemma}[Density matrix derivatives]\label{thm:densitymatder}
Consider a matrix-valued function $B$ depending on a complex parameter such that $B(t) = B_0+B_1t$, where $B_0,B_1$ are Hermitian.
Let $X,\Lambda$ correspond to a parameter dependent diagonalization for a sufficiently small neighborhood of  $t_0=0$, i.e., 
\begin{equation*}
  X(t)\Lambda(t)X(t)^H=B(t),
\end{equation*}
with $X(t)^HX(t)=I$ 
and $\Lambda(0)=\operatorname{diag}(\lambda_1(0),\ldots,\lambda_n(0))$,
where $\lambda_1(0)<\cdots< \lambda_p(0)<\lambda_{p+1}(0)<\cdots< \lambda_n(0)$. 
Let $X(t)\in\CC^{n\times n}$
be decomposed as
$X(t)=[X_1(t),X_2(t)]$,
where  $X_1(t)\in\CC^{n\times p}$
and $P(t):=X_1(t)X_1(t)^H$.
Then, 
\begin{equation*}
 \vect(P'(0))= -(\overline{X(0)}\otimes X(0))D(X(0)^T\otimes X(0)^H)\vect(B_1)
\end{equation*}
where
\begin{equation}\label{eq:Ddef}
D:=\diag(\vect(R)).  
\end{equation}
\end{lemma}
\begin{proof}
Let us consider the domain $\mathcal{D} = \left(-\infty,\lambda_{mid}\right)\cup \left(\lambda_{mid},\infty\right)$, where $\lambda_{mid} = \frac{\lambda_p(0)+\lambda_{p+1}(0)}{2}$ and define the step function $h:\mathcal{D}\to \mathbb{R}$, which acts as a filter for the $p$ smallest eigenvalues of $B(0) = B_0$:
\begin{equation}\label{eq:heavyside}
h(x) = \begin{cases}1,\quad x<\lambda_{mid}\\0,\quad x>\lambda_{mid}\end{cases}.
\end{equation}
We can generalize $h:M\to M$ as a matrix function (in the sense of \cite{HIGHAMFOM}) and rewrite the density matrix function as
\begin{equation}\label{eq:densityheavy}
P(t) = X_1(t)X_1(t)^H = X(t)\begin{bmatrix}I_p& 0\\ 0& 0\\ \end{bmatrix}X(t)^H = X(t)h(\Lambda(t))X(t)^H = h(B(t)),
\end{equation}
where we assume that $t$ lies in a sufficiently small neighbourhood of $t_0=0$ such that $\lambda_p(t)<\lambda_{mid}<\lambda_{p+1}(t)$.
Then,
\begin{eqnarray}\label{eq:dfrechet}
P'(0) =& \lim\limits_{\epsilon\to 0} &\frac{h(B(\epsilon))-h(B(0))}{\epsilon} \nonumber\\
      =& \lim\limits_{\epsilon\to 0, \norm{E}\to 0} &\frac{X(0)\Big(h(\Lambda(0)+X(0)^HEX(0))-h(\Lambda(0))\Big)X(0)^H}{\epsilon},
\end{eqnarray}
where 
\begin{equation}\label{eq:eb0}
E = \epsilon B_1.
\end{equation}
If we denote by $L(F,G)$ the Fr\'echet derivative of $h$ evaluated at $F$ and applied to $G$, then from \eqref{eq:dfrechet}, we get
\begin{align}\label{eq:dfrechet2}
P'(0) &=& \lim\limits_{\epsilon\to 0, \norm{E}\to 0} &\frac{X(0)\Big(L\big(\Lambda(0),X(0)^HEX(0)\big)+O(\norm{E}^2)\Big)X(0)^H}{\epsilon}\nonumber\\
      &=& \lim\limits_{\epsilon\to 0, \norm{E}\to 0} &\frac{X(0)L\big(\Lambda(0),X(0)^HEX(0)\big)X(0)^H}{\epsilon}\nonumber\\
      &=& &X(0)L\big(\Lambda(0),X(0)^HB_1X(0)\big)X(0)^H&
\end{align}
The last equation is a consequence of \eqref{eq:eb0} and the fact that $O(\norm{E}^2)/\epsilon$ goes to zero as $\epsilon$ goes to zero and $L$ being linear in the second argument. From the Dalecki{\u i}-Kre{\u i}n theorem\cite[Theorem ~3.11]{HIGHAMFOM},\cite{daleckiikrein}, we have
\begin{equation}\label{eq:Lfrechet}
L(\Lambda(0),X(0)^HB_1X(0)) = U\circ (X(0)^HB_1X(0)),
\end{equation}
where $U$ denotes the matrix of divided differences
\begin{equation*}
U := \begin{cases}\frac{h(\lambda_i(0))-h(\lambda_j(0))}{\lambda_i(0)-\lambda_j(0)}&i \neq j\\h'(\lambda_i(0))&i=j.\end{cases}
\end{equation*}
By definition of $h$ in \eqref{eq:heavyside} and from \eqref{eq:Rdef}, we see that $U$ reduces to $-R$.
Vectorizing \eqref{eq:dfrechet2} and using \eqref{eq:Lfrechet} leads us to
\begin{align}
\vect(P'(0)) &= -\vect\Big(X(0)\big(R\circ (X(0)^HB_1X(0))\big)X(0)^H\Big)\\ \nonumber
             &= -(\overline{X(0)}\otimes X(0))D(X(0)^T\otimes X(0)^H)\vect(B_1),
\end{align} 
which is a result of repeated application of the matrix product vectorization identity $\vect(LMN) = (N^T\otimes L)\vect(M)$.
\end{proof}
In practice, the computation of the density matrix as in \eqref{eq:psifp} can be done by using \eqref{eq:densityheavy}. Moreover, in SCF iterations, variations of the problem can be solved, where an approximation of the step function is used. A common choice for this approximation is the Fermi-Dirac distribution: $f_{\mu,\beta}(t) = \frac{1}{1+e^{\beta(t-\mu)}}$, where the parameter $\mu$ is usually selected such that $\trace(P(0)) = p$. The function $f_{\mu,\beta}$ tends to the step function in the limit $\beta \to 0$.
Note that Lemma \ref{thm:densitymatder} can be generalized for such an approximation of the density matrix $P_f$ as follows:
\begin{equation*}
\vect(P'_f(0)) = (\overline{X(0)}\otimes X(0))D_f(X(0)^T\otimes X(0)^H)\vect(B'(0)),
\end{equation*}
where $D_f = \diag(\vect(R_f))$ and 
\begin{equation*}
R_f = \begin{cases}\frac{f(\lambda_i(0))-f(\lambda_j(0))}{\lambda_i(0)-\lambda_j(0)}&i \neq j\\f'(\lambda_i(0))&i=j\end{cases}.
\end{equation*}

\begin{thm}[Density matrix local convergence]\label{thm:densitymat}
Let $P_*=X_{1*}X_{1*}^H\in\CC^{n\times n}$ be
a fixed point of $\Psi$, i.e., $P_*=\Psi(P_*)$. 
Then, the SCF iteration satisfies 
\[
\vech(P_{k+1}-P_*)=
   \vech(\Psi(P_k)-P_*)=J_{P}\vech(P_k-P_*)+O(\|\vech(P_k-P_*)\|^2)
\]
where
\begin{equation}  
  J_P=-T(\overline{X_*}\otimes X_*)
  D({X_*}^T\otimes {X_*}^H)L'  \label{eq:JPdef}
\end{equation}
and $L'\in\CC^{n^2\times m}$ is defined by
\begin{equation}\label{eq:Lprime}
L'=
\left(\vect(\mathcal{L}(\vech^{-1}(e_1)),\ldots,
\vect(\mathcal{L}(\vech^{-1}(e_{m})))\right),
\end{equation}
$X_* = [X_{1*},X_{2*}]$ and $e_1,e_2,\ldots,e_m$ are the first $m$ columns of the identity matrix $I_n$.
\end{thm}
\begin{proof}
Applying the operator $\vech(\cdot)$ to our iteration, we get
\begin{equation*}
  \begin{aligned}
  \vech(P_{k+1}) &= \vech(\Psi(P_k))\\
                 &= \vech(\vect^{-1}(\vect(\Psi(\vech^{-1}(\vech(P_k)))))).
  \end{aligned}
\end{equation*}
Hence, the fixed point iteration can be re-written as
\begin{equation*}
\vech(P_{k+1}) = f(\vech(P_{k})),
\end{equation*} 
where $f: \mathbb{R}^{m} \to \mathbb{R}^{m}$,
\begin{equation*}
  f(v) = \vech(\vect^{-1}(\vect(\Psi(\vech^{-1}(v)))))\quad \forall v \in \mathbb{R}^m.
\end{equation*}
A Taylor expansion around the fixed-point $\vech(P_*)$ gives us,
\begin{equation*}
  \vech(P_{k+1}-P_*) = J_P\vech(P_k-P_*) + O(\|\vech(P_k-P_*)\|^2),
\end{equation*}
where $J_P$ is the Jacobian of $f$ evaluated in $\vech(P_*)$.
The $j$-th column of $J_P$ is given by
\begin{equation}\label{eq:jach}
\begin{aligned}
J_P(:,j) &= \lim_{\epsilon \to 0} \frac{f(\vech(P_*)+\epsilon e_j) - f(\vech(P_*))}{\epsilon}\\
         &= \lim_{\epsilon \to 0} \frac{\vech(\vect^{-1}(\vect(\Psi(P_*+\epsilon \vech^{-1}(e_j))-\Psi(P_*))))}{\epsilon}.\\
\end{aligned}
\end{equation}
By using linearity of the vectorization operators and of $\mathcal{L}$,
we can now invoke Lemma~\ref{thm:densitymatder} with $B(\alpha) = A+\mathcal{L}(P_*+\alpha \vech^{-1}(e_j))$ to get
\begin{multline*}
  \lim_{\epsilon \to 0} \frac{\vect(\Psi(P_*+\epsilon \vech^{-1}(e_j))-\Psi(P_*))}{\epsilon} =\\
  (\overline{X_*}\otimes X_*)D({X_*}^T\otimes {X_*}^H)\vect(\mathcal{L}(\vech^{-1}(e_j))).
\end{multline*}
Using this and equation \eqref{eq:jach}, we have
\begin{equation}
\begin{aligned}
J_P(:,j) &= \vech(\vect^{-1}((\overline{X_*}\otimes {X_*})D({X_*}^T\otimes {X_*}^H)\vect(\mathcal{L}(\vech^{-1}(e_j)))))\\
&= T(\overline{X_*}\otimes {X_*})D({X_*}^T\otimes {X_*}^H)\vect(\mathcal{L}(\vech^{-1}(e_j))). \label{eq:JPX}
\end{aligned}
\end{equation}
Due to the fact that 
$\vect(\mathcal{L}(\vech^{-1}(e_j)))$ is the only component in \eqref{eq:JPX} depending
on $j$, we obtain \eqref{eq:JPdef} by factorizing the matrix $T(\overline{X_*}\otimes {X_*})D({X_*}^T\otimes {X_*}^H)$.
\end{proof}
\newpage
\section{Convergence factor bounds and their interpretation}
\subsection{Spectral-norm bounds}
Since \eqref{eq:psifp} is a nonlinear fixed-point map and the Jacobian is evaluated at a fixed point in Theorem~\ref{thm:densitymat}, the convergence factor is
\begin{equation*}
  c=\rho(J_P) 
\end{equation*}
where $\rho(J_P)$ denotes the spectral radius of $J_P$. 
Moreover, a sufficient and almost necessary condition for local convergence is $c < 1$.
Due to the fact that the spectral radius is smaller
than any operator norm, we have in particular for
the spectral norm:
\begin{equation*}
  c\leq \norm{J_P}_2 =:c_2.
\end{equation*}
\subsubsection{Naive bounds}
Now note that  $\overline{X_*}\otimes X_*$ and $X_*^T\otimes X_*^H$ are orthogonal matrices,
and that $D$ defined by \eqref{eq:Ddef} is a diagonal matrix whose largest element is the reciprocal gap, such
that
\[
\|D\|_2=\max_{i}|d_{i,i}| =\frac{1}{\delta}.
\]
By using this and the Cauchy-Schwartz inequality we obtain 
a straight-forward upper bound for $c$, 
\begin{equation}\label{eq:mnaive}
    c\leq \norm{T}_2 \norm{\overline{X}\otimes {X}}_2 \norm{D}_2 \norm{({X}^T\otimes {X}^H)}_2 \norm{L'}_2 \leq \frac{\norm{L'}_2}{\delta} := c_{\rm naive},
\end{equation}
where we dropped subscript $*$ in the eigenvector matrix $X_*$ for notational convenience. We can conclude from \eqref{eq:mnaive}  that a small gap implies a larger value of the upper bound $c_{naive}$, indicating slow convergence. This is consistent with the  well-known fact that problems with a small gap are more difficult to solve using the SCF iteration which is concluded in several convergence analysis works, e.g. \cite{Yang:2009:SCF}. Note that the bound \eqref{eq:mnaive} does not depend on the gap alone but also on $\norm{L'}_2$, which can be large and difficult to analyze. The matrix $L'$ depends on the action of the operator $\mathcal{L}$, which leads us to the pursuit of other bounds which may quantify this dependence in a way that is easier to interpret.
\subsubsection{Cycled permutation}
Different bounds can be derived by using the fact that
the spectral radius does not change
when we reverse the order of multiplication of matrices, i.e.,
$\rho(AB)=\rho(BA)$.
Therefore, from the definition of $J_P$ and
\[
c=\rho(J_P) = \rho(T(\overline{X}\otimes X)D({X}^T\otimes {X}^H)L'),
\]
we obtain variants based on cyclic permutation
\begin{subequations}\label{eq:cyclic1}
\begin{eqnarray}
  c&=&\rho((\overline{X}\otimes X)D({X}^T\otimes {X}^H)L'T)\\
  &=&\rho(D({X}^T\otimes {X}^H)L'T(\overline{X}\otimes X)).
\end{eqnarray}
\end{subequations}
Both equations in \eqref{eq:cyclic1} lead to the bound
\begin{equation*}
   c\le  \|D({X}^T\otimes {X}^H)L'T\|_2=:c_{2,a}.
\end{equation*}
The cyclic permutation can be continued such that
\begin{subequations}\label{eq:cyclic2}
\begin{eqnarray}
  c&=&\rho(({X}^T\otimes {X}^H)L'T(\overline{X}\otimes X)D)\\
  &=&\rho(L'T(\overline{X}\otimes X)D({X}^T\otimes {X}^H)).
\end{eqnarray}
\end{subequations}
Equation \eqref{eq:cyclic2} leads to the bound
\begin{equation}  \label{eq:c2b}
   c\le  \|L'T(\overline{X}\otimes X)D\|_2=:c_{2,b}.
\end{equation}
In the following we need the symmetrization operator formally
defined as 
\begin{equation}\label{eq:symmet}
S(X):=\sum_{j=1}^m\vech(X)_j\vechinv(e_j)
\end{equation}
or equivalently $S(L+D+R)=L+D+L^T$, where $L+D+R$ is the
decomposition into the lower triangular, diagonal and upper triangular matrices.
Using \eqref{eq:symmet}, we have the identity
\begin{eqnarray}\label{eq:lpident}
  L'\vech(X) &=& \sum_{j=1}^m\vect(\LLL(\vechinv(e_j)))\vech(X)_j  \\
  &=&\vect(\mathcal{L}(S(X))), \nonumber
\end{eqnarray}
due to the definition of $L'$ in \eqref{eq:Lprime} and the linearity of $\LLL$.
The columns of the matrix in \eqref{eq:c2b}
can be expressed as 
\begin{eqnarray*}
(L'T(\overline{X}\otimes X)D)_{:,j}&=&L'T\vect(x_\ell x_m^H)d_{j,j}\\
&=&L'\vech(x_\ell x_m^H)d_{j,j}
\end{eqnarray*}
where $j=n(m-1)+\ell$ and where we used \eqref{eq:vvech}
in the last step.
Hence, from \eqref{eq:lpident} we have
\begin{equation} \label{eq:c2b_columns}
(L'T(\overline{X}\otimes X)D)_{:,j}=\vect(\LLL(S(x_\ell x_m^H)))d_{j,j}.  
\end{equation}
The columns of the matrix inside the norm in \eqref{eq:c2b} is given by \eqref{eq:c2b_columns}. This formula can be interpreted as follows. We clearly see that the action of the linear operator applied to the outer products of eigenvectors $\LLL(S(x_\ell x_m^H))$ has significance. The weighting with $d_{j,j}$ implies that only pairs of eigenvectors of different occupancy are relevant (in this bound). This quantity is further described in the following section.
\subsection{Higher gaps}
We begin by decomposing the matrix obtained by cycled permutation in \eqref{eq:cyclic2} as follows,
\begin{multline} \label{eq:Dsep}
L'T(\overline{X}\otimes X)D({X}^T\otimes {X}^H)=\\
L'T(\overline{X}\otimes X)(D-D_1-\cdots-D_{2q})({X}^T\otimes {X}^H)
+L'T(\overline{X}\otimes X)(D_1+\cdots+D_{2q})({X}^T\otimes {X}^H)
\end{multline}
where $D_j$ are rank one diagonal matrices (and we take $2q$ terms for symmetry reasons). We have $q=p(n-p)$ unique gaps and they occur twice each on the diagonal of $D$. Hence, this decomposition reveals the dependence of the convergence on eigenvalue gaps other than just the smallest gap $\delta$. In the following theorem we quantify this dependence along with the dependence on the outer products of the eigenvectors $\LLL(S(x_\ell x_m^H))$ as introduced in the previous subsection.

In the formulation of the theorem, we use the set $\Omega_q\subset [1,n]\times [1,n]$ which contains the indices of $R$ that have entries corresponding to the $q$ smallest gaps. As a result, $\Omega_q$ contains $2q$ elements.
\begin{figure}[h]
  \begin{center}
    \scalebox{0.6}{\includegraphics{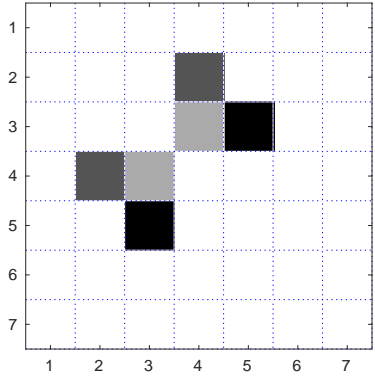}}
    \caption{
      Schematic illustration of elements of $\Omega_3$ as indices of $R$ for the real-valued problem in subsection \ref{subsec:dlap} with $n=7, p=3, \alpha = 10.0$.
    }
    \label{fig:schema}
  \end{center}
  
\end{figure}
In figure~\ref{fig:schema}, we visualize the elements of $\Omega_3$. The set $\Omega_3$ comprises the indices of the reciprocal gap matrix $R$ which correspond to these gaps, that is, $\Omega_3 = \{(4,3),(3,4),(4,2),(2,4),(5,3),(3,5)\}$.
\begin{theorem}[Higher gaps]\label{theo:hogaps}
  The convergence factor of the SCF-iteration is bounded by
  \begin{equation}  \label{eq:hgaps}    
    \rho(J_p)\le \frac{\|L'\|_2}{\delta_{q+1}}+
    \sum_{(\ell,m)\in\Omega_q}\frac{1}{|\lambda_\ell-\lambda_m|}\|\LLL(S(x_\ell x_m^H))\|_F := c_{gap,q}
  \end{equation}
     for any $q\in [0,p(n-p)]$, where $\delta_{p(n-p)+1}:=\infty$ and $\delta_1 = \delta$.
  \end{theorem}

\begin{proof}
  For notational convenience, we express $D$ in terms of $R$ (defined in \eqref{eq:Rdef}),
  i.e.,
  \begin{eqnarray*}
    D&=&\diag(\vect(R))=\sum_{\ell,m}^n r_{\ell,m}\vect(e_\ell e_m^T)\vect(e_\ell e_m^T)^T   \\
    &=&\sum_{(\ell,m)\not\in\Omega_q} r_{\ell,m}\vect(e_\ell e_m^T)\vect(e_\ell e_m^T)^T+   \\
    &&\sum_{(\ell,m)\in\Omega_q} r_{\ell,m}\vect(e_\ell e_m^T)\vect(e_\ell e_m^T)^T.
  \end{eqnarray*}
  The idea of the proof is to take the last sum in this equation as $\sum_{j=1}^{2q} D_j$
  with the decomposition in \eqref{eq:Dsep}.
By the triangle inequality, we have
\begin{equation}  \label{eq:gapbound_first}
   c \leq \|L'T(\overline{X}\otimes X)(D-D_1-\cdots-D_q)\|
+\sum_{j=1}^{2q}\|L'T(\overline{X}\otimes X)D_j\|.
\end{equation}
The first term in \eqref{eq:gapbound_first} is of the form used in the naive bound \eqref{eq:mnaive}, except that the diagonal matrix
is modified by setting the contribution corresponding to the $q$ first gaps to zero. We obtain directly the first term in \eqref{eq:hgaps}, 
\[
\|L'T(\overline{X}\otimes X)(D-D_1-\cdots-D_{2q})\|\le
\frac{\|L'\|_2}{\delta_{q+1}}.
\]
The second term in \eqref{eq:gapbound_first} is 
\[
\sum_{j=1}^{2q}\|L'T(\overline{X}\otimes X)D_j\|=
\sum_{(\ell,m)\in\Omega_q} r_{\ell,m}\|L'T(\overline{X}\otimes X)\vect(e_\ell e_m^T)\vect(e_\ell e_m^T)^T\|.
\]
This can be simplified with the  identity \eqref{eq:c2b_columns} which implies that
\begin{subequations}\label{eq:gapbound_second}
\begin{eqnarray}
\|L'T(\overline{X}\otimes X)\vect(e_\ell e_m^T)\vect(e_\ell e_m^T)^T\|&=&\|\vect(\LLL(S(x_\ell x_m^H)))\vect(e_\ell e_m^T)^T\| \\ 
&=&\|\vect(\LLL(S(x_\ell x_m^H)))\| \\ 
&=&\|\LLL(S(x_\ell x_m^H))\|_F.
\end{eqnarray}
\end{subequations}
The last two equalities follow from the fact that the 
spectral norm of a matrix with one non-zero column, is the
two-norm of that column vector, which is the Frobenius
norm of the corresponding matrix.
The proof is concluded by combining \eqref{eq:gapbound_first}
with \eqref{eq:gapbound_second} and
noting that $r_{\ell,m}=\frac{1}{|\lambda_\ell-\lambda_m|}$.

\end{proof}
Theorem \ref{theo:hogaps} should be further interpreted as follows. The parameter $q$ is free and the theorem therefore provides us with a family of bounds parameterized by $q$. For example, $q=0$ gives us $\rho(J_P) \leq \frac{\norm{L'}_2}{\delta_1}$, which is the naive bound from \eqref{eq:mnaive}. For $q = 1$, we have
\begin{equation*}
\rho(J_p) \leq \frac{\norm{L'}_2}{\delta_2}+\frac{\norm{\mathcal{L}(S(x_px_{p+1}^H))}_F+\norm{\mathcal{L}(S(x_{p+1}x_p^H))}_F}{\delta_1}.
\end{equation*}
By induction, $q=k$ gives us a bound that is a function of the $k+1$ smallest gaps and the norm of the action of $\mathcal{L}$ on the outer products of eigenvector-pairs corresponding to the gap indices of the $k$ smallest gaps. 
\subsection{Illustrative example}\label{sec:acadexample}
In order to illustrate the insight provided by Theorem~\ref{theo:hogaps}, we provide an example showing how in certain situations, using a bound with a higher value of $q$ provides us tighter upper bounds. Consider the following problem parameterized by $\epsilon$.
\begin{equation}\label{eq:acadproblem}
\left(\begin{bmatrix}0 & \epsilon & 0\\ \epsilon & 1+d & \epsilon\\ 0 & \epsilon & 10\\\end{bmatrix} + \begin{bmatrix}1 & 0 & 0\\ 0 & 1 & 0\\ 0 & 0 & 100\\\end{bmatrix} \circ x_1x_1^H\right)X = X\Lambda.
\end{equation}
Let $d=0.16$. In the notation in equation \eqref{eq:ap}, we have,
\[
A_0 = \begin{bmatrix}0 & \epsilon & 0\\ \epsilon & 1+d & \epsilon\\ 0 & \epsilon & 10\\\end{bmatrix}\quad \LLL(P) = L\circ P = \begin{bmatrix}1 & 0 & 0\\ 0 & 1 & 0\\ 0 & 0 & 100\\\end{bmatrix} \circ P .
\]
First note that for $\epsilon = 0$, the solution is given by
\[
\Lambda(0) = \begin{bmatrix}\lambda_1& 0& 0\\0& \lambda_2& 0\\0& 0& \lambda_3\\\end{bmatrix} =\begin{bmatrix}1 & 0 & 0\\ 0 & 1+d & 0\\ 0 & 0 & 10\\\end{bmatrix},\quad X(0) = \begin{bmatrix}1 & 0 & 0\\ 0 &1 & 0\\ 0 & 0 & 1\\\end{bmatrix}.
\]%
\begin{figure}[h!]
  \centering
  \scalebox{1}{\includegraphics{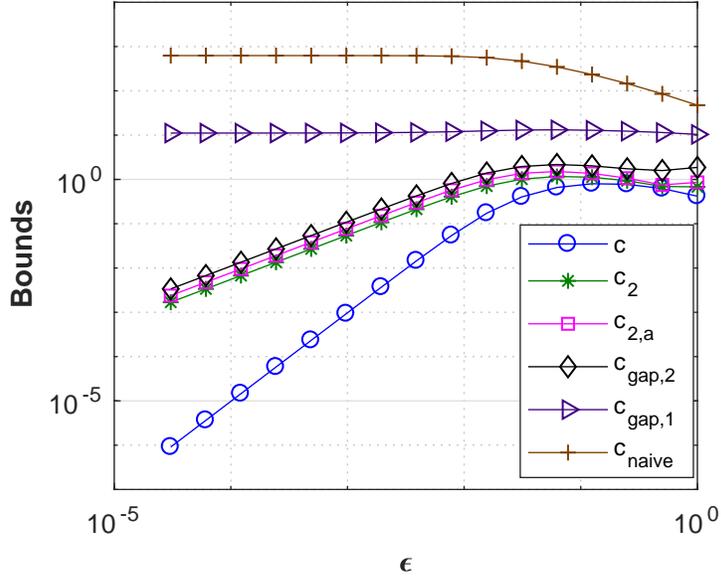}}
  \caption{Convergence factor and bounds for the illustrative example}
  \label{fig:acadbounds}
\end{figure}%
\begin{figure}[h!]
  \centering
  \scalebox{1}{\includegraphics{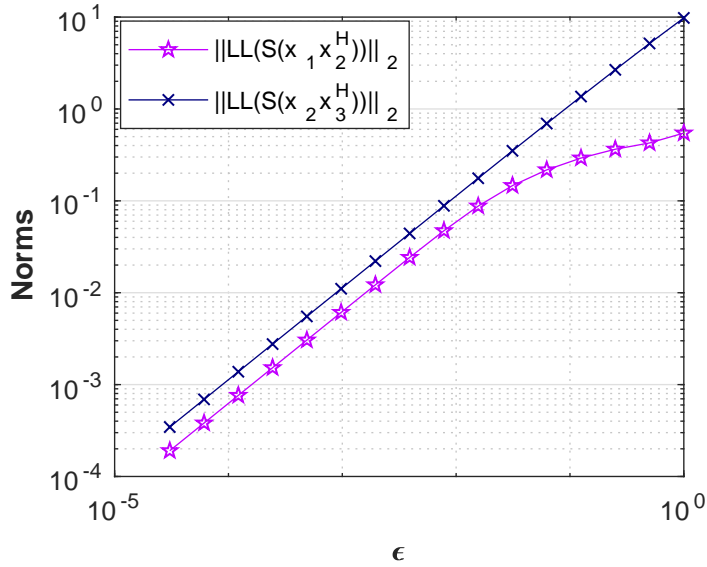}}
  \caption{Norm of $\mathcal{L}(x_1x_2^H)$ and $\mathcal{L}(x_2x_3^H)$}
  \label{fig:actionofl}
\end{figure}%
\begin{figure}[h!]
  \centering
  \scalebox{1}{\includegraphics{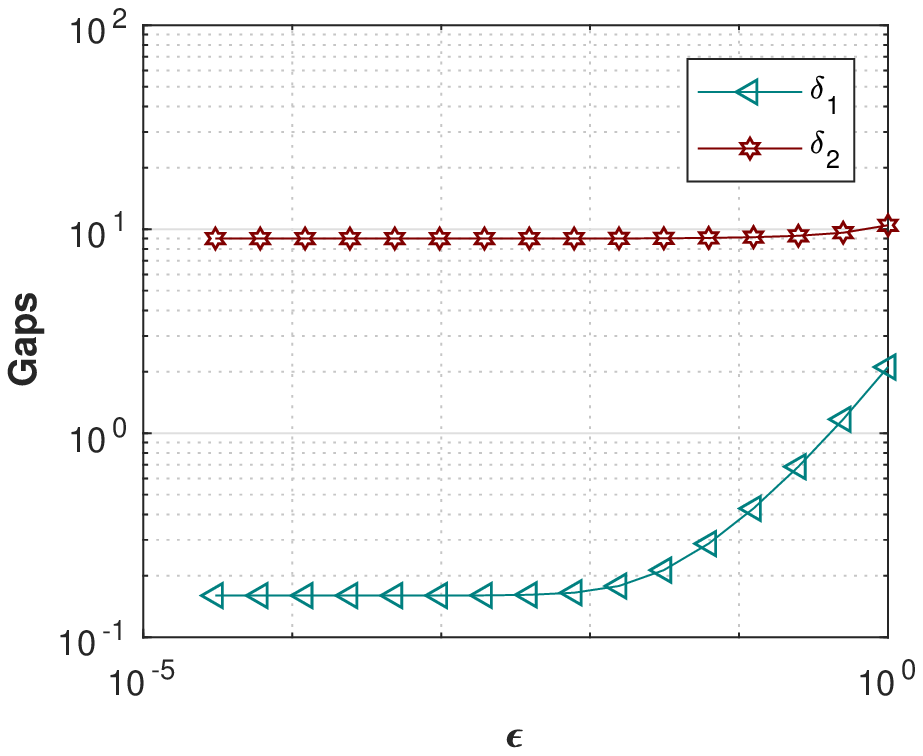}}
  \caption{Variance of $\delta_1$ and $\delta_2$ with $\epsilon$}
  \label{fig:gaps}
\end{figure}%
Varying $\epsilon$, solving the resulting problem instances and plotting the convergence factor and its bounds gives us figure~\ref{fig:acadbounds}. In figure~\ref{fig:actionofl}, we plot the norm of $L'$ along with the norm of action of $\mathcal{L}$ on the outer products of the eigenvectors ($x_1x_2^H$ and $ x_1x_3^H$) with $\epsilon$. In figure~\ref{fig:gaps}, we visualize how the gaps $\delta_1$ and $\delta_2$ vary with $\epsilon$. Note that we have $\norm{L'}_2 = 100$, which is independent of $\epsilon$.
As can be seen from figures ~\ref{fig:acadbounds} and ~\ref{fig:gaps}, $c_{naive}$ shows a direct inverse dependence on $\delta_1$, which is expected from \eqref{eq:mnaive} because $\norm{L'}$ is constant.

We can see from figures \ref{fig:acadbounds} and \ref{fig:actionofl} that $c_{naive}$ is not a good approximation to the convergence factor $c$, in comparison to our bounds $c_{gap,2}$, $c_2$ and $c_{2,a}$. The increase in $c$ as we increase $\epsilon$ (which means slower convergence for larger values of $\epsilon$) is not captured by $c_{naive}$ or $c_{gap,1}$ as they are essentially constant for small $\epsilon$. However, from figure~\ref{fig:acadbounds}, we see that the increase in $c$ coincides with the increase of the norm of action of $\LLL$ on the outer products of the eigenvectors ($x_1x_2^H$ and $ x_1x_3^H$), as seen in figure~\ref{fig:actionofl}. This behaviour is better captured in the formula for the upper bound $c_{gap,2}$,
\begin{align*}
c_{gap,2} &= \frac{\norm{\mathcal{L}(S(x_1x_{2}^H))}_F+\norm{\mathcal{L}(S(x_{2}x_1^H))}_F}{\delta_1}\\
          &+\frac{\norm{\mathcal{L}(S(x_1x_{3}^H))}_F+\norm{\mathcal{L}(S(x_{3}x_1^H))}_F}{\delta_2}.
\end{align*}
Although the bounds $c_{gap,2}$ , $c_2$ and $c_{2,a}$ are better approximations of $c$ as compared to $c_{naive}$, there is still a discrepancy in the slopes in figure~\ref{fig:acadbounds}, and the rate of increase of $c$ is faster than that of the bounds.
We now provide a more detailed analysis. First note that differentiating \eqref{eq:acadproblem} with respect to $\epsilon$, and setting $\epsilon = 0$, we obtain
\begin{equation}\label{eq:xp0}
X'(0) = \begin{bmatrix}0& \frac{1}{\lambda_2-\lambda_1}& 0\\ \frac{1}{\lambda_1-\lambda_2}& 0& \frac{1}{\lambda_3-\lambda_2}\\0& \frac{1}{\lambda_2-\lambda_3}& 0\end{bmatrix},\quad \Lambda'(0) = 0,\quad D'(0) = 0.
\end{equation} 
Let $J(\epsilon)$ denote the parameter dependent Jacobian evaluated at the solution,
\begin{equation*}
  J(\epsilon) = -T\left(X(\epsilon)\otimes X(\epsilon)\right)D(\epsilon)\left(X(\epsilon)^H\otimes X(\epsilon)^H\right)L'.
\end{equation*}
Differentiating w.r.t $\epsilon$, setting $\epsilon = 0$, and using $X(0) = I$,
\begin{equation}\label{eq:jprime0}
  \begin{split}
    J'(0) = -T\bigg(&\left(X'(0)\otimes I+I\otimes X'(0)\right)D(0)+D'(0)\\
    &- D(0)\left(X'(0) \otimes I+I\otimes X'(0)\right)\bigg)L'.\\
  \end{split}
\end{equation}
Using the formulae from \eqref{eq:xp0} and substituting into \eqref{eq:jprime0}, we get,
\begin{equation*}
J'(0) = \frac{1}{{(\lambda_2-\lambda_1)}^2}\begin{bmatrix}
        -e_2& 0& 0& e_2& 0& 0
        \end{bmatrix}.
\end{equation*}
From the structure of $J'(0)$ and the fact that all eigenvalues of $J'(0)$ are zero, we have, 
\begin{eqnarray}\label{eq:rhojp0}
\rho(J'(0)) &=& 0,\\
\norm{J'(0)}_2 &=& \frac{1}{{(\lambda_2-\lambda_1)}^2}\quad. \nonumber
\end{eqnarray}
This allows us to carry out a Taylor series analysis for $\rho\left(J(\epsilon)\right)$ and $\norm{J(\epsilon)}_2$ around 0,
\begin{equation}\label{eq:rhoj}
  c = \rho(J(\epsilon)) = \rho(J(0)+\epsilon J'(0)+\mathcal{O}(\epsilon^2)) \approx \epsilon\rho(J'(0))+\mathcal{O}(\epsilon^2) = \mathcal{O}(\epsilon^2).
\end{equation}
Similarly,
\begin{equation}\label{eq:normj}
  c_2 = \norm{J(\epsilon)}_2 = \epsilon\norm{J'(0)}_2+\mathcal{O}(\epsilon^2) = \frac{\epsilon}{{(\lambda_2-\lambda_1)}^2}+\mathcal{O}(\epsilon^2) = \mathcal{O}(\epsilon).
\end{equation}
Hence, from \eqref{eq:rhoj} and \eqref{eq:normj}, we expect $c$ to vary at a rate that is an order of magnitude faster than $c_2$ for very small values of $\epsilon$, which is exactly what we observe in figure~\ref{fig:acadbounds}. We clearly see from  \eqref{eq:rhojp0},\eqref{eq:rhoj} and \eqref{eq:normj} that this is because $J'(0)$ has zero eigenvalues (and hence zero spectral radius) but non-zero norm. This illustrates how the two-norm based bounds can overestimate $c$.
\section{Numerical examples}
\subsection{Discrete Laplacian example}\label{subsec:dlap}
In this subsection, we apply our theory to a minor variation of the problem type discussed in \cite[Section~5]{liu2014convergence}. In the context of this paper, this translates to 
\[
A_0 = \begin{bmatrix}
      2/h^2& -1/h^2+i/2h& 0& \ldots & 0\\
      -1/h^2-i/2h& 2/h^2& -1/h^2+i/2h& \ddots& \vdots \\
      0& -1/h^2-i/2h& 2/h^2& \ddots& 0\\
      \vdots& \ddots& \ddots& \ddots& -1+i/2h\\
      0& \ldots& 0& -1/h^2-i/2h& 2/h^2\\
      \end{bmatrix},
\]
which is the discretized 1D differential operator: $\dfrac{\partial^2}{\partial x^2}+i\dfrac{\partial}{\partial x}$. In a PDE setting, this would correspond to a diffusion term added with a complex convection term discretized with a central difference scheme with grid spacing $h$. We also have
\[
\mathcal{L}(P) = \alpha Diag\left(Re(A_0)^{-1}diag\left(P\right)\right).
\]
Note that $\mathcal{L}(\cdot)$ depends only on the diagonal of $P$. 
\begin{figure}
\subfigure[Convergence factor and bounds]{{\includegraphics[height=5.cm,width=6.2cm]{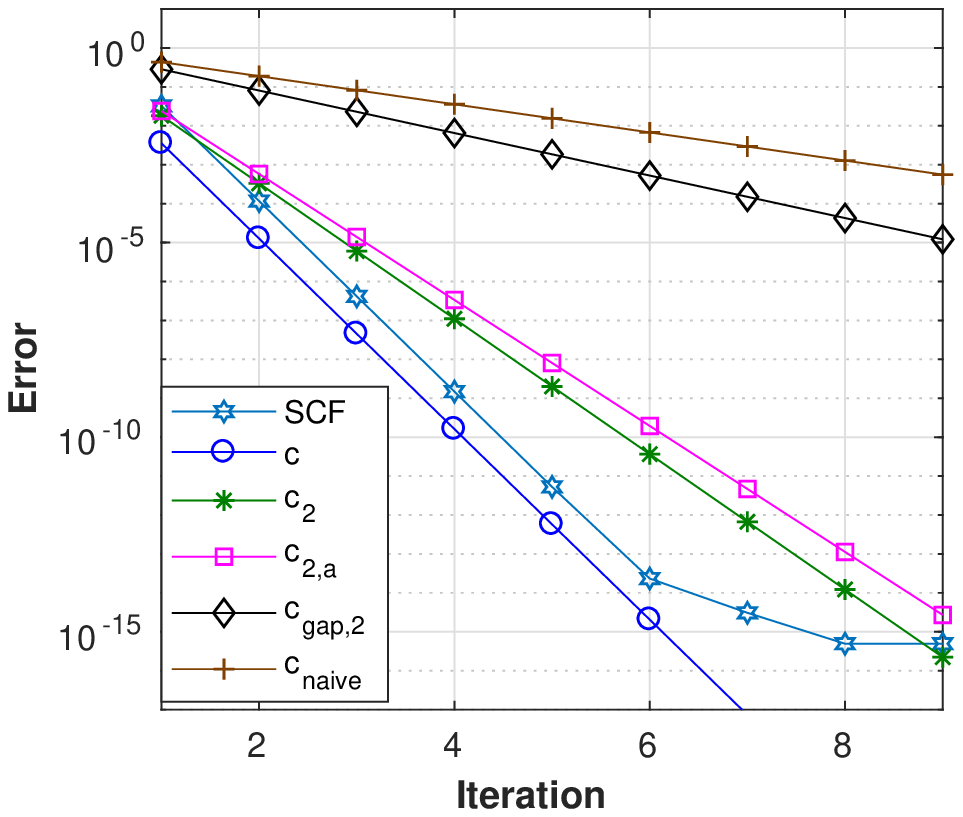}}}
\subfigure[Distribution of eigenvalues]{{\includegraphics[height=5.cm,width=6.2cm]{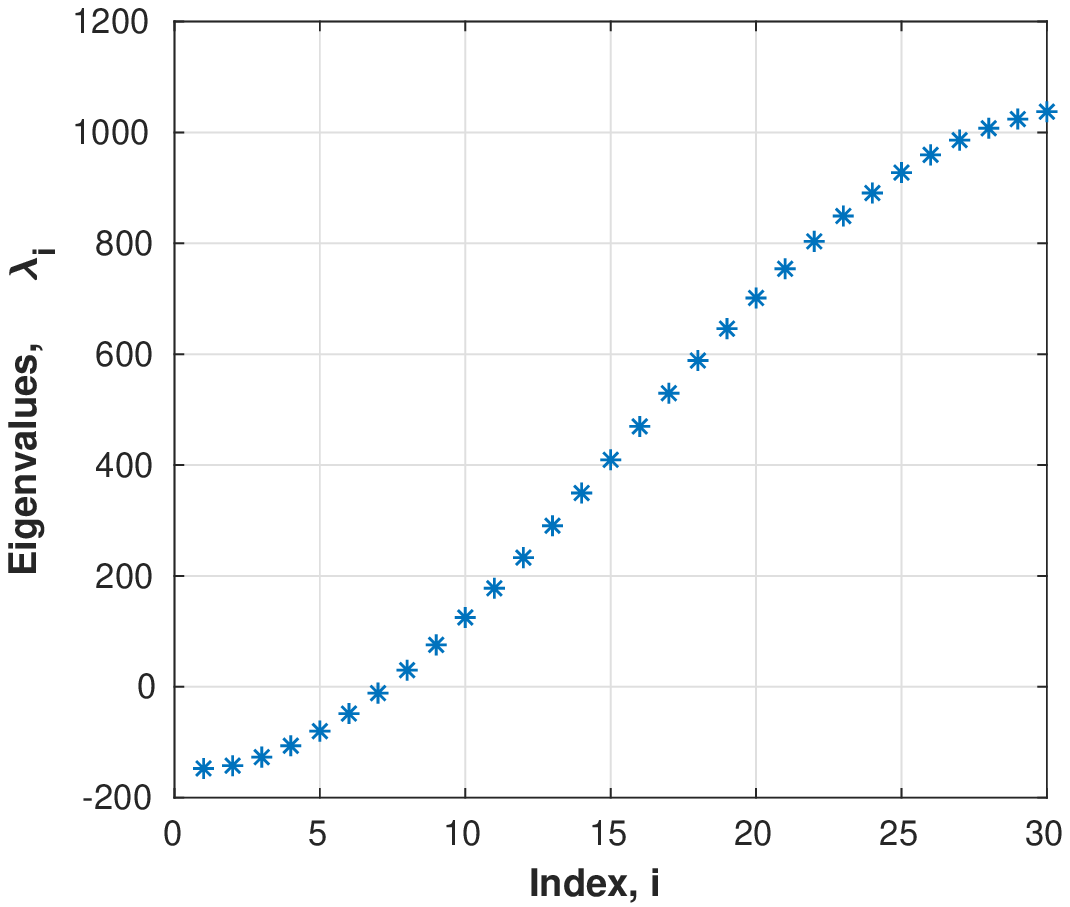}}}
\subfigure[Variance with $n$]{{\includegraphics[height=5.cm,width=6.2cm]{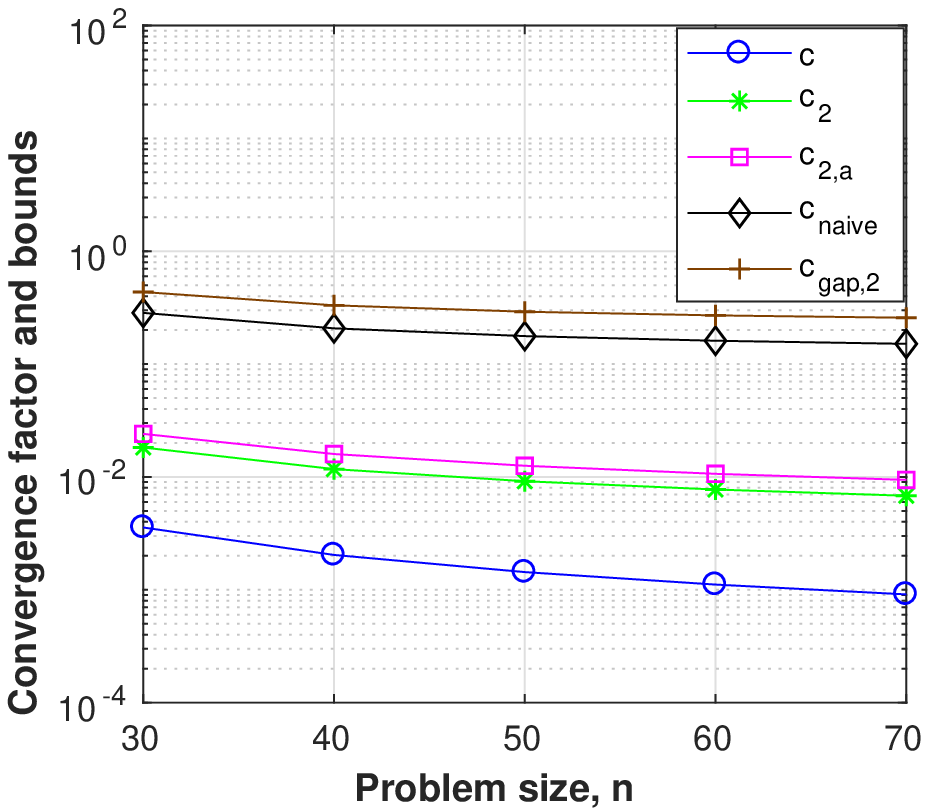}}}
\subfigure[Variance with $\alpha$]{{\includegraphics[height=5.cm,width=6.2cm]{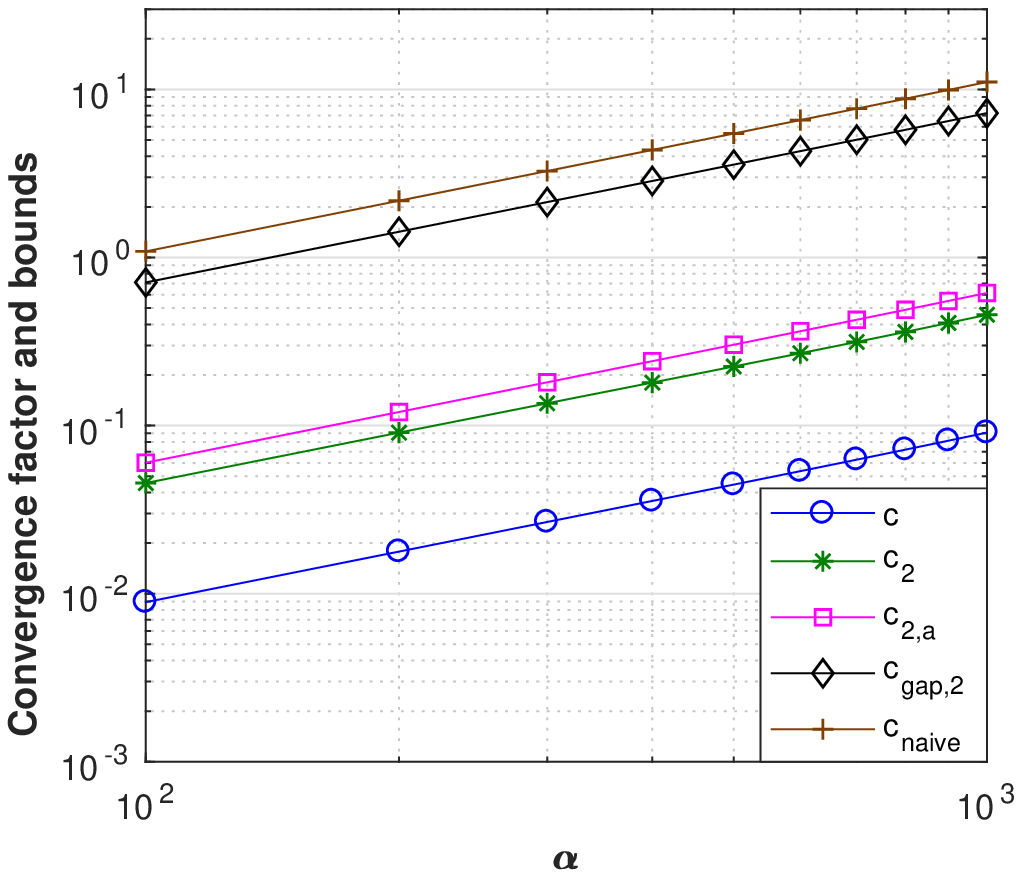}}}
\caption{Complex-valued problem for $n=30$\big((a),(b) and (d)\big), $p= 15, \alpha=40.0$\big((a),(b) and (c)\big)}
\label{fig:cproblem}
\end{figure}
Figure~\ref{fig:cproblem}(a) shows us that the predicted convergence rate $c$ agrees perfectly with SCF convergence history. The norm based bounds $c_2$ and $c_{naive}$ are slightly worse than the exact rate $c$. As expected, $c_{naive}$ is the least accurate upper bound, but $c_{gap,2}$ is only slightly better. As seen from figure~\ref{fig:cproblem}(b)  the gaps between the eigenvalues are not very well separated, that is, the higher gaps are not much larger than $\delta$. More precisely, $\delta_3$ is not much larger than $\delta_1$ and $\delta_2$. This explains why $c_{gap,2}$ is not a good approximation to the norm based bounds for this problem. Figure~\ref{fig:cproblem}(d) shows a linear increase in the value of the convergence factor and the upper bounds with $\alpha$, which is expected from the linear dependence of $L'$ on $\alpha$ and equation \eqref{eq:cyclic1}. From figure~\ref{fig:cproblem}(c), we also see that convergence becomes faster(that is $c$ decreases) with increase in problem size for a constant value of $\alpha$ and $p$.
To make a comparison of our upper bounds with the convergence factor derived from \cite[Theorem~4.2]{liu2014convergence}, we change the problem by setting
\[
A_0 = \begin{bmatrix}
      2/h^2& -1/h^2& 0& \ldots & 0\\
      -1/h^2& 2/h^2& -1/h^2& \ddots& \vdots \\
      0& -1/h^2& 2/h^2& \ddots& 0\\
      \vdots& \ddots& \ddots& \ddots& -1/h^2\\
      0& \ldots& 0& -1/h^2& 2/h^2\\
      \end{bmatrix}
\]
since the analysis in that paper is presented for real-valued problems. The operator $\mathcal{L}(\cdot)$ is the same as before. We plot the upper bound $c_{Liu} = \frac{2\alpha\sqrt{n}\norm{A_0^{-1}}_2}{\delta_1}$ along with the other upper bounds for this modified real-valued problem in figure~\ref{fig:realconv}. 
\begin{figure}[h!]
  \centering
  \scalebox{0.7}{\includegraphics{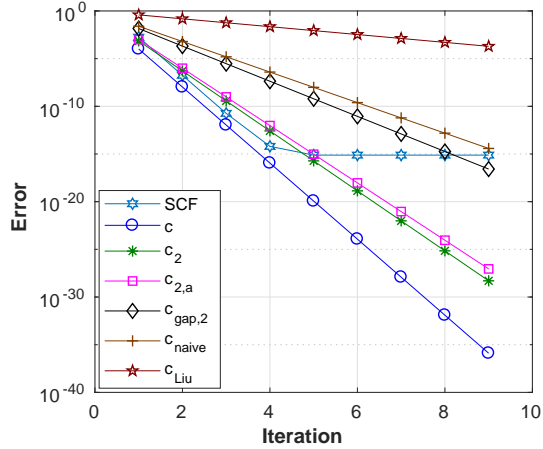}}
  \caption{Real-valued problem for $n = 60$, $\alpha = 5.0$, $p = 25$}
  \label{fig:realconv}
\end{figure}
Figure~\ref{fig:realconv} suggests to us that the upper bounds discussed in this paper are improvements over the upper bound in \cite{liu2014convergence}.
\subsection{Water molecule example}\label{subsec:watermolecule}
In this example, we apply the SCF iteration to a problem that originates from the modelling of a water molecule system. 
The discretization involves a restricted Hartree-Fock approximation with a set of $n=13$ basis functions and $p=5$.
For any nonlinear eigenvalue problem that results from a Hartree-Fock approximation, we have
\[
A(X_1X_1^H) = H_{core}+2G(R^{-1}X_1X_1^H{R^{-1}}^H)
\]
where $G(\cdot)$ is a linear operator and $H_{core}$ is a sum of two matrices that correspond to terms for kinetic energy and the nuclear-electron interaction energy. In our context, $\mathcal{L}(P) = 2G(R^{-1}P{R^{-1}}^H)$. Here, $R$ is a lower triangular matrix that results from a cholesky decomposition of the "overlap matrix". The overlap matrix is hermitian and obtained by computing integrals of products of basis functions, as explained in \cite[Section~2.4]{Rudberg}. For the purpose of reproducibility, we provide the coordinates of the nuclei of the Oxygen and Hydrogen atoms in the following table. Note that all data is in atomic units.
\begin{center}
\begin{tabular}{|c|c|c|c|c|}
\hline
Atom& Charge($\mathbf{e}$)& x($\mathbf{a_0}$)& y($\mathbf{a_0}$)& z($\mathbf{a_0}$)\\
\hline
O& 8.0& 0.0& 0.0& 0.0\\
\hline
H& 1.0& -1.809& 0.0& 0.0\\
\hline
H& 1.0& 0.453549& 1.751221& 0.0\\
\hline
\end{tabular}
\end{center}
The computation was performed using Ergo\cite{ERGO-JCTC-2011,ERGO-SOFTWAREX}, which is a software package for large-scale SCF calculations. The standard Gaussian basis set 3-21G was used and the starting guess for the density matrix was projected from the calculation using a smaller STO-3G basis set.
We plot the SCF convergence history and the exact convergence factor $c$. We have not plotted the other upper bounds that we derived because they overestimate $c$ by a large margin. Instead, based on the theory in Theorem~\ref{theo:hogaps} we investigate the bound that neglects certain terms such that 
\[
\widetilde{c_{2}} = \norm{T(\overline{X}\otimes X)D_1({X}^T\otimes {X}^H)L'}_2
\]
which is the spectral norm of a 2-rank approximation of the Jacobian $J_P$ (taking into account only the entries that contain $\delta_1$ in $D$). As we can see from figure ~\ref{fig:watemolecule}, the observed behaviour of SCF convergence agrees with that predicted by the exact value of the spectral radius, $c$. 
\begin{figure}[h!]
  \subfigure[Convergence factor and bounds]{{\includegraphics[height=5.cm,width=6.2cm]{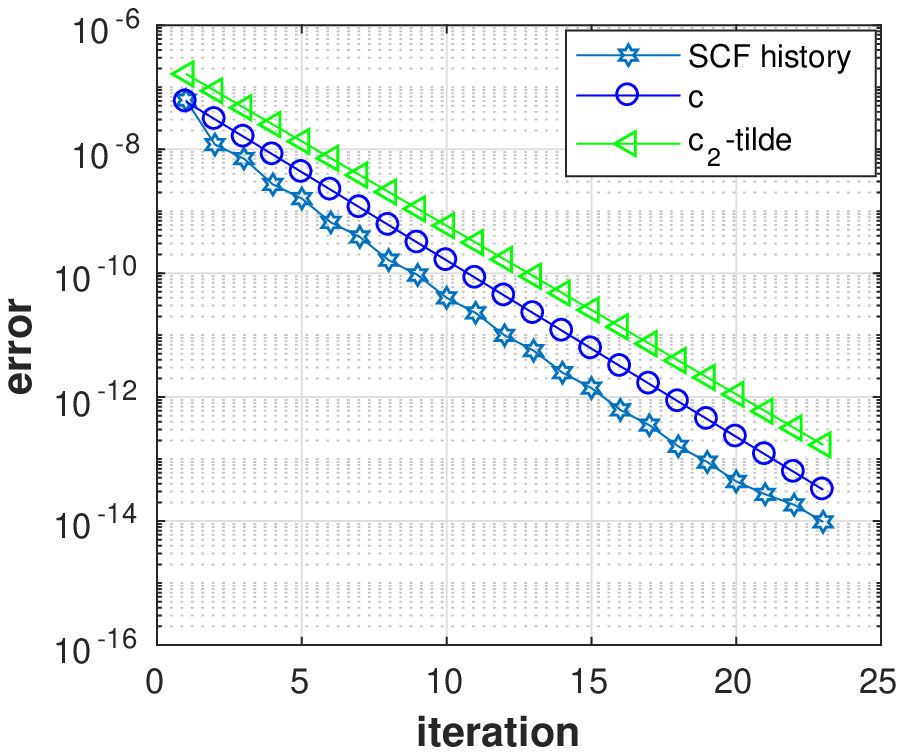}}}
  \subfigure[Distribution of eigenvalues]{{\includegraphics[height=5.cm,width=6.2cm]{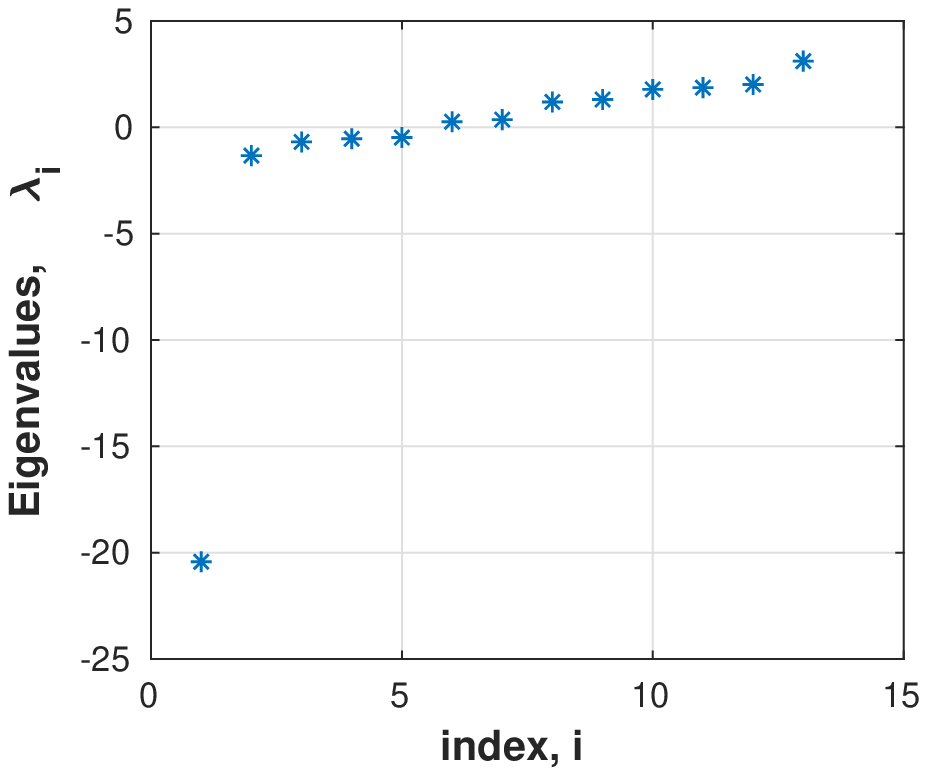}}}

  \caption{Water molecule problem with $n=13, p=5$}
  \label{fig:watemolecule}
\end{figure}
\section{Conclusions and outlook}
The SCF algorithm is an important algorithm in many fields. We have provided a new convergence characterization for the algorithm using a density matrix based analysis of a fixed point map. The upper bounds derived for the spectral radius of the Jacobian of the fixed point map illustrate how the convergence depends on the different problem parameters and physical properties. In particular, Theorem~\ref{theo:hogaps} provides a mathematical footing for studying how the gaps interact with the outer products of eigenvectors to affect the convergence properties.  This is a quantification of Stanton's observation in \cite[Section~IV]{Stanton:1981:CLOSEDSHELL}, where he points out that typically, divergence in SCF calculations is not due to a single very low energy excitation, but to the interaction of several moderately low excitations. The discussion of the illustrative example in section \ref{sec:acadexample} explains how when the Hessian has zero spectral radius but non-zero norm, an upper bound based on the interaction of higher gaps is needed to give a more accurate picture of the convergence behaviour. Finally, the application of our upper bounds to practical problems in subsections ~\ref{subsec:dlap} and ~\ref{subsec:watermolecule} reveal that our bounds are slightly better approximations to the convergence factor than the bounds that exist in previous literature. 
\bibliography{eliasbib}
\bibliographystyle{AIMS}
\medskip

\end{document}